\documentclass[UTF-8,reqno]{amsart}

\usepackage{txfonts}
\usepackage{enumerate, bbm}
\usepackage{amssymb,url,color, booktabs}
\usepackage{mathrsfs}
 \usepackage{scalerel} 
\usepackage{threeparttable}
\usepackage{makecell}

\usepackage{stfloats}
\usepackage{pgfplots}
\pgfplotsset{compat=1.18}

\usepackage{color}
\usepackage[colorlinks=true]{hyperref}
\hypersetup{
    linkcolor=blue,          
    citecolor=red,        
    filecolor=blue,      
    urlcolor=cyan
}
\usepackage{tikz} 
\usepackage{color}
\definecolor{MyDarkBlue}{cmyk}{0.8,0.3,0.8,0.4}
\definecolor{yellow}{rgb}{0.99,0.99,0.70}
\definecolor{white}{rgb}{1.0,1.0,1.0}
\definecolor{black}{rgb}{0.00,0.00,0.00}
\usepackage[nobysame,abbrev]{amsrefs}
\renewcommand{\eprint}[1]{{\it Available at}\href{https://arxiv.org/abs/#1}{\it{ arXiv:#1}}.}
\renewcommand{\PrintDOI}[1]{\url{https://doi.org/#1}}
\renewcommand{\MR}[1]{\href{https://mathscinet.ams.org/mathscinet-getitem?mr=#1}{\color{cyan}{MR#1}}}

\BibSpec{article}{%
+{}{\PrintAuthors} {author}
+{,}{ \textrm} {title}
+{:}{ \textrm} {subtitle}
+{.}{ \textit} {journal}
+{,}{ \textbf} {volume}
+{}{ \parenthesize} {date}
+{,}{ } {pages}
+{}{ \parenthesize} {language}
+{.}{} {transition}
+{.}{ \PrintReviews} {review}
+{.}{ \eprint} {eprint}
+{.}{ \PrintDOI} {doi}
}
\BibSpec{book}{%
+{}{\PrintAuthors} {author}
+{,}{ \textit} {title}
+{:}{ \textit} {subtitle}
+{.}{ \textrm} {series} 
+{,}{ Vol.} {volume} 
+{.}{ } {publisher}
+{,}{ } {date}
+{}{ \parenthesize} {language}
+{.}{} {transition}
+{.}{ \PrintReviews} {review}
+{.}{ \PrintDOI} {doi}
}

\numberwithin{equation}{section}

\newcommand{\be}{\begin{eqnarray}}
\newcommand{\ee}{\end{eqnarray}}
\newcommand{\ce}{\begin{eqnarray*}}
\newcommand{\de}{\end{eqnarray*}}
\newtheorem{theorem}{Theorem}[section]
\newtheorem{lemma}[theorem]{Lemma}
\newtheorem{remark}[theorem]{Remark}
\newtheorem{definition}[theorem]{Definition}
\newtheorem{proposition}[theorem]{Proposition}

\newtheorem{corollary}[theorem]{Corollary}

\def\eps{\varepsilon}

\def\e{\mathrm{e}}

\def\p{\partial}

\def\[{{\Big[}}
\def\]{{\Big]}}
\def\<{{\langle}}
\def\>{{\rangle}}
\def\({{\big(}}
\def\){{\big)}}

\def\sgn{\mbox{\rm sgn}}

\def\dif{{\mathord{{\rm d}}}}

\def\min{{\mathord{{\rm min}}}}

\def\bb2{{\boldsymbol{2}}}

\def\={&\!\!=\!\!&}

\def\bB{{\mathbf B}}
\def\bC{{\mathbf C}}

\def\cR{{\mathcal R}}

\def\mB{{\mathbb B}}

\def\mE{{\mathbb E}}

\def\mN{{\mathbb N}}

\def\mP{{\mathbb P}}

\def\mR{{\mathbb R}}

\def\b1{{\mathbbm 1}}

\def\sB{{\mathscr B}}

\def\sF{{\mathscr F}}

\def\sS{{\mathscr S}}

\def\geq{\geqslant}
\def\leq{\leqslant}
\def\ge{\geqslant}
\def\le{\leqslant}

\def\div{\mathord{{\rm div}}}

\def\eps{\varepsilon}

\def\e{\mathrm{e}}

\def\p{\partial}

\def\[{{\Big[}}
\def\]{{\Big]}}
\def\<{{\langle}}
\def\>{{\rangle}}

\def\sgn{\mbox{\rm sgn}}

\def\dif{{\mathord{{\rm d}}}}

\def\min{{\mathord{{\rm min}}}}

\def\={&\!\!=\!\!&}
\def\bt{\begin{theorem}}
\def\et{\end{theorem}}
\def\bl{\begin{lemma}}
\def\el{\end{lemma}}
\def\br{\begin{remark}}
\def\er{\end{remark}}
\def\bd{\begin{definition}}
\def\ed{\end{definition}}
\def\bp{\begin{proposition}}
\def\ep{\end{proposition}}
\def\bc{\begin{corollary}}
\def\ec{\end{corollary}}

\def\geq{\geqslant}
\def\leq{\leqslant}
\def\ge{\geqslant}
\def\le{\leqslant}

\def\div{\mathord{{\rm div}}}

 \def\R{\mathbb R}
 \def\R{\mathbb R}

\def\<{\langle} \def\>{\rangle}

\allowdisplaybreaks

\begin{document}

\title
[Unified Condition of weak and strong well-posedness for MVSDEs Driven by $\alpha$-Stable Processes]{Strong and Weak Well-Posedness of McKean-Vlasov SDEs driven by $\alpha$-stable processes Under Unified Condition}
\author{Zimo Hao}

\thanks{\it Keywords: $\alpha$-stable process; McKean-Vlasov SDEs; regularization by noise.  }

\address{
Zimo Hao: Fakult\"at f\"ur Mathematik, Universit\"at Bielefeld, 33615, Bielefeld, Germany,
Email: zhao@math.uni-bielefeld.de}

\thanks{Zimo Hao is grateful for the DFG through the CRC 1283/2 2021 - 317210226
``Taming uncertainty and profiting from randomness and low regularity in analysis, stochastics and their applications''.  }

\begin{abstract}
In this paper, we consider $\alpha \in (0,2)$ and establish the strong well-posedness of McKean--Vlasov SDEs driven by an $\alpha$-stable process with a H\"older (Besov) kernel $K \in \bC^\beta$, where $\beta > 1-\alpha$. This condition coincides with the well-known threshold for the weak well-posedness.
\end{abstract}

\maketitle

\setcounter{tocdepth}{1}
\tableofcontents

\section{Introduction}
Consider the following SDE driven by an $\alpha$-stable process:
\begin{align}\label{SDE}
\dif X_t = b(X_t)\dif t + \dif L_t,
\end{align}
where $b: \mathbb{R}^d \to \mathbb{R}^d$ with $b \in \bC^\beta$, where $\bC^\beta$ denotes the Besov space $\bB^\beta_{\infty,\infty}$ (see Definition \ref{iBesov} below). The process $(L_t)_{t \ge 0}$ is a rotationally invariant $\alpha$-stable L\'evy process with $\alpha \in (0,2)$, whose infinitesimal generator is the fractional Laplacian $\Delta^{\frac{\alpha}{2}} := -(-\Delta)^{\frac{\alpha}{2}}$ (see \eqref{fL} below), defined on a complete filtered probability space $(\Omega, \sF, \mathbb{P}, (\sF_t)_{t \ge 0})$.
 It is well known (see, e.g., \cite{CZZ21}) that:
\begin{itemize}
\item[(i)] If $\beta > 1 - \alpha$, there exists a unique weak solution.
\item[(ii)] If $\beta > 1 - \frac{\alpha}{2}$, there exists a unique strong solution.
\end{itemize}

\begin{remark}
Let $\alpha \in (1,2)$. In case (i), \cite{CZZ21} considers only the case $\beta > 0$. For $\beta \in \left(\frac{1-\alpha}{2}, 0\right)$, the existence and uniqueness of weak and martingale solutions have been established in \cite{LZ22, CM22}, while the existence and uniqueness of strong solutions in the one-dimensional case have been shown in \cite{ABM20}. For $\beta \in \left(\frac{3(1-\alpha)}{2}, \frac{1-\alpha}{2}\right]$ and enhanced drifts $b$, a unique generalized martingale solution was shown in \cite{KP22} via paracontrolled calculus. Furthermore, for $\beta \in (1-\alpha, \frac{1-\alpha}{2}]$ and $\div b \in \bC^\beta$, a unique weak solution was obtained in \cite{HW23}.
\end{remark}

The condition $\beta > 1 - \alpha$ arises naturally from scaling arguments (see \eqref{scaling} below), whereas the stronger assumption $\beta > 1 - \frac{\alpha}{2}$ is primarily technical conditions required for PDE estimates in Zvonkin-type transformations (cf. \cite{Pr12, CSZ18, Zh21, CZZ21, WH23}). This discrepancy between the H\"older conditions for weak and strong solutions motivates the central question of this paper:
\begin{align}\label{Question}
\text{\it Can one establish the strong well-posedness under the assumption $\beta > 1 - \alpha$?} \tag{Q}
\end{align}

When $\alpha = 2$, this question has been resolved for the McKean–Vlasov SDE (MVSDE):
\begin{align*}
\dif X_t = (K * \mu_{X_t})(t,X_t), \dif t + \dif W_t,
\end{align*}
where $K: \mathbb{R}_+ \times \mR^d \to \mR^d$ is a measurable function, $\mu_{X_t}$ is the time marginal law of the solution, and $W$ is a standard $d$-dimensional Brownian motion. Specifically, \cite{CJM25} established pathwise uniqueness and strong well-posedness for $K \in L^\infty(\mathbb{R}_+; \bC^{\beta})$ with $\beta > -1$.

However, for $\alpha < 2$ and the MVSDE:
\begin{align}\label{DDSDE}
\dif X_t = (K * \mu_{X_t})(t,X_t), \dif t + \dif L_t,
\end{align}
\cite{CJM25} only addresses the case $\alpha > 1$ and proves the strong well-posedness under the condition
\begin{align*}
\beta > 2 - \tfrac{3}{2}\alpha,
\end{align*}
which is strictly stronger than $\beta > 1 - \alpha$.

In this paper, we aim to answer question \eqref{Question} for the MVSDE \eqref{DDSDE}. Throughout, we assume $\alpha \in (0,2)$, $T > 0$, and:
\begin{align}\label{con}
K \in L^\infty_T\bC^\beta := L^\infty([0,T]; \bC^\beta), \quad \text{for some } \beta \in \mR. \tag{\bf H$_\beta$}
\end{align}
Under the assumption $\beta > 1 - \alpha$, it is well-known that the MVSDE \eqref{DDSDE} admits a unique global weak solution (see \cite{CJM25} for the case $\alpha > 1$ and \cite{HRW24} for $\alpha \leq 1$).

\vspace{1mm}

Here is our main result.

\begin{theorem}\label{in:main}
Let $\alpha \in (0,2)$ and assume condition \eqref{con} holds with some $\beta > 1 - \alpha$. Then \eqref{DDSDE} admits a unique global strong solution for any initial date $X_0\in\sF_0$.
\end{theorem}

\begin{remark}
This result fully answers question \eqref{Question} for the MVSDE \eqref{DDSDE} with $\alpha\in(0,2)$.
\end{remark}

\br
Let us now return to the original question \eqref{Question} for the SDE \eqref{SDE}. This case is significantly more challenging than the McKean–Vlasov setting, as there is no mechanism to transfer regularity from $\mu_{X_t}$ to the drift term $K * \mu_{X_t}$ (see \eqref{02} below). In fact, even in the Brownian motion case ($\alpha = 2$), question \eqref{Question} is only solved for $\beta \ge 0$ (with $\beta = 0$ interpreted as $\bC^0 = L^\infty$). When $\beta < 0$, strong well-posedness is only known in one dimension, and only for $\beta > -\tfrac{1}{2}$ (see \cite{BC02}). For $\beta > -1$, only weak well-posedness is known (see \cite{HZ23, GP24}).
\er


\vspace{2mm}

\leftline{\it\bf Scaling analysis:} 

\vspace{2mm}

Let $X$ be a solution to the SDE
$$
\dif X_t=B(t,X_t)\dif t+\dif L_t,
$$
where $B\in L^q(\mR_+;\bC^\beta)$ for some $q\in[1,\infty]$ and $\beta\in\mR$.
For any $\lambda > 0$, define the rescaled processes and drift by  $$
X_t^\lambda := \lambda^{-1} X_{\lambda^\alpha t},  \ 
L_t^\lambda := \lambda^{-1} L^{(\alpha)}_{\lambda^\alpha t},  \
B_\lambda(x) := \lambda^{\alpha - 1} B(\lambda^{\alpha} t,\lambda x).
$$  
Formally, the rescaled process $X^\lambda$ satisfies the equation
$$
\mathrm{d} X_t^\lambda = B_\lambda(t,X_t^\lambda) \, \mathrm{d}t + \mathrm{d}L_t^\lambda.
$$  
Moreover, the scaling behavior of the drift in the $L^q(\mR_+;\bC^{-\beta})$ norm satisfies  
$$
\|B_\lambda\|_{L^q(\mR_+;\bC^{\beta})} \approx \lambda^{\alpha - 1 -\frac{\alpha}{q}+ \beta} \|B\|_{L^q(\mR_+;\bC^{\beta})}.
$$ 
As $\lambda \to 0$, this leads to a classification of three scaling regimes based on the behavior of the norm: 
\begin{align}\label{scaling}
    \text{\it Scaling Subcritical: } \tfrac{\alpha}{q}+\beta < \alpha - 1;  
    \end{align}
and
$$
\text{\it Scaling Critical: } \tfrac{\alpha}{q}+\beta = \alpha - 1; \quad  
\text{\it Scaling Supercritical: } \tfrac{\alpha}{q}+\beta > \alpha - 1.
$$
In the subcritical regime, we have $ \lim_{\lambda\to0}\|B_\lambda\|_{L^q(\mR_+;\bC^{\beta})}=0$, implying that the rescaled drift becomes negligible, and thus the solution behaves like the driving noise: $X\approx L$. 

Notably, when $B(t,x)=B(x)$, there exists a well-known counterexample due to \cite{TTW74}, which shows that in one dimension, for any $\beta\in(0,1-\alpha)$ (i.e., in the supercritical regime), there exists a drift $b \in \bC^{-\beta}$, for which both pathwise uniqueness and uniqueness in law fail for SDE \eqref{SDE}.
\vspace{2mm}

\vspace{2mm}

\leftline{\it\bf Sketch of the proof of Theorems \ref{in:main}:}
\vspace{2mm}
 For two solutions $X^1$ and $X^2$ to \eqref{DDSDE} with the same $\alpha$-stable process $L$ and initial data $X^1_0=X^2_0$, $\mP$-a.s. By the well-known weak well-posedness, we have $\mu_{X_t^1}=\mu_{X_t^2}$. Hence we only need to construct a unique strong solution to the SDE:
 \begin{align}\label{SDE2}
     \dif X_t=B(t,X_t)\dif t+\dif L_t,
 \end{align}
where $B(t,x):=(K*\mu_{X}^1)(t,x)=(K*\mu_{X}^2)(t,x)$.
 
Based on the Yamada-Watanabe theorem, to show the strong well-posedness, we only need to show the pathwise uniqueness of the SDE \eqref{SDE2}. We note that
$$
\|B\|_{L^\infty_T \bC^\beta}\lesssim \|K\|_{L^\infty_T \bC^\beta}
$$
with the same parameter $\beta>1-\alpha$, with which we cannot show the pathwise uniqueness directly for the SDE \eqref{SDE2}. However, the time marginal law $\mu_t:=\mu_{X^1_t}$ satisfies the following nonlinear Fokker-Planck equation (FPE):
\begin{align}\label{in:FPE}
    \p_t \mu_t=\Delta^{\frac\alpha2}\mu_t-\div ((K*\mu_t) \mu_t).
\end{align}
It is well-known that under condition \eqref{con}, $\mu_t$ admits a density $\rho_t$ with respect to the Lebesgue measure (see \cite{CJM25, HRW24}). By the scaling, it is expected that we can have for any $\eps>0$,
\begin{align}\label{in:FPEe}
    \|\rho_t\|_{\bB^{\frac{\alpha}{2}-\eps}_{1,\infty}}\lesssim t^{-\frac{1}{2}+\frac{\eps}{\alpha}} \in L^2[0,T],
\end{align}
for any $T>0$. Then the convolution inequality (see \eqref{ineq:con} below) yields that
\begin{align}\label{02}
\|B\|_{L^2_T\bC^{\beta+\frac\alpha2-\eps}}=\|K*\rho\|_{L^2([0,T];\bC^{\beta+\frac\alpha2-\eps})}\lesssim \left(\int_0^T\|\rho_t\|_{\bB^{\frac{\alpha}{2}-\eps}_{1,\infty}}^2\dif t\right)^{\frac12}\|K\|_{L^\infty_T \bC^\beta}<\infty.
\end{align}
For sufficiently small $\varepsilon>0$, we have $\gamma:=\beta+\frac\alpha2-\eps>1-\frac\alpha2$ under the assumption $\beta>1-\alpha$.
The condition $\gamma>1-\tfrac{\alpha}{2}$ allows us to apply the classical approach based on Zvonkin’s transformation (see \cite{CZZ21, SX23, HRW24}). The key step is to derive suitable estimates for the solution $u$ of the backward Kolmogorov equation
\begin{align*}
    \p_t u+\Delta^{\frac{\alpha}{2}}u+B\cdot\nabla u+B=0,\quad u(T)=0,
\end{align*}
in particular to control
$\|u\|_{L^2_T\bC^{1+\frac{\alpha}{2}}}$. This quantity can be bounded by $\|u\|_{L^2_T \bC^{\alpha+\gamma}}$, which follows from the Schauder theory.

Moreover, the assumption $B \in L^2_T\bC^{\gamma}$ remains in the scaling-subcritical regime \eqref{scaling}, since
$$
\tfrac{\alpha}{2}-\gamma=\tfrac{\alpha}{2}-(1-\tfrac\alpha2+\eps)=\alpha-1-\eps<\alpha-1.
$$
Thus, it is natural that pathwise uniqueness for \eqref{SDE2} can be obtained under such a condition on $B$. In fact, a rigorous proof of pathwise uniqueness for $B\in L^2_T\bC^{\gamma}$ with $\gamma>1-\tfrac{\alpha}{2}$ has already been given in \cite{TW25}. Therefore, the main task is to prove \eqref{in:FPEe}.

\vspace{1mm}

When $\alpha>1$, the estimates in \eqref{in:FPEe} were obtained in \cite{HRZ23}, which directly yields Theorem \ref{in:main}.

For the case $\alpha<1$, which is called the supercritical case, it becomes a bit difficult since in this case the diffusion part $\Delta^{\frac\alpha2}\mu_t$ is weaker than the conservation term $\div((K*\mu_t)\mu_t)$. In the paper, we would like to use the energy method in \cite{CZZ21, SX23} which was used to estimates the solution to corresponding PDE with the supercritical case $\alpha\in(0,1)$. However, we need the $L^1$-estimates for the density with block operator $\cR_j\rho_t$, which can not be obtained with energy method, which is available for $L^p$-estimates with $p\ge2$ (see \cite[Theorem 3.3]{CZZ21}). To this end, we consider the following non-divergence form PDE with any fixed $t\ge 0$:
 \begin{align*}
     \p_s u^t(s)=\Delta^{\frac{\alpha}{2}}u^t(s)+(K(t-s)*\rho_{t-s})\cdot \nabla u^t(s),\quad u^t(0)=\varphi\in \bB^{-\delta}_{\infty,\infty}
 \end{align*}
 with some $\delta>0$.
 Applying It\^o's formula to $s\to u^t(t-s,X_s)$, one sees that
 \begin{align*}
     \langle\rho_t,\varphi\rangle=\mE \varphi(X_t)=\mE u^t(t,X_0),
 \end{align*}
 which implies that
 \begin{align*}
    \|\rho_t\|_{\bB^{\delta}_{1,1}}\lesssim\sup_{\|\varphi\|_{\bB^{-\delta}_{\infty,\infty}}=1} |\langle\rho_t,\varphi\rangle|\le \sup_{\|\varphi\|_{\bB^{-\delta}_{\infty,\infty}}=1} \|u^t(t)\|_\infty.
 \end{align*}
 Therefore, we only need to show the following estimate, which bypassed the $L^1$-estimates:
 \begin{align*}
 \|u^t(t)\|_\infty\lesssim t^{-\frac{\delta}{\alpha}}\|\varphi\|_{\bB^{-\delta}_{\infty,\infty}}.    
 \end{align*}

\vspace{2mm}

\leftline{\it\bf Structure of the paper:}
\vspace{2mm}
The paper is organized as follows. In Section \ref{Sec:2}, we introduce the definitions and basic concepts of Besov spaces, as well as the $\alpha$-stable process. In Section \ref{Sec:3}, we introduce a recent strong well-posedness result for SDE \eqref{SDE2}. In Section \ref{Sec:4}, we establish the regularity estimate for the solution to nonlinear FPE \eqref{in:FPE}, which implies the estimate \eqref{in:FPEe}. Finally, we show the proof of our main result, Theorem \ref{in:main} in Section \ref{Sec:5}.

\vspace{2mm}

\leftline{\it\bf Conventions and notations:}
\vspace{2mm}

Throughout this paper, we use the following conventions and notations: As usual, we use $:=$ as a way of definition. Define $\mN_0:= \mN \cup \{0\}$ and $\mR_+:=[0,\infty)$. 
We use $A \asymp B$ and $A\lesssim B$ 
to denote $c^{-1} B \leq A \leq c B$ and $A \leq cB$, respectively, for some unimportant constant $c \geq 1$. We also use $A   \lesssim_\Theta B$ to denote $A \leq c(\Theta) B$ when we want to emphasize that the implicit constant $c$ depends on parameters $\Theta$. 

For every $p\in [1,\infty)$, we denote by $L^p$ the space of all $p$-order integrable functions on $\mR^d$ with the norm denoted by $\|\cdot\|_p$. For $p=\infty$, we set
\begin{align*}
    \|f\|_\infty:=\sup_{x\in\mR^d}|f(x)|.
\end{align*}
We use $C_b^\infty$ to denote the space of infinitely differentiable bounded functions, and $C^\infty_0$ to denote the space of infinitely differentiable functions with compact support. We also let $C_0$ denote the space of continuous functions vanishing at infinity.

For a Banach space $\mB$ and $T>0$, $q\in[1,\infty]$, we denote by
$$
L_T^q\mB:= L^q([0,T];\mB).
$$
 
 



\section{Preliminaries}\label{Sec:2}
\subsection{Besov spaces}
In this section, we introduce the definition and some basic properties of Besov spaces. Let $\sS(\mR^d)$ be the Schwartz space of all rapidly decreasing functions on $\mR^d$, and $\sS'(\mR^d)$ be
the dual space of $\sS(\mR^d)$ called Schwartz generalized function (or tempered distribution) space. Given $f\in\sS(\mR^d)$, 
the Fourier transform $\hat f$ and the inverse Fourier transform  $\check f$ are defined by
$$
\hat f(\xi) :=(2 \pi)^{-d/2}\int_{\mR^d} \e^{-i\xi\cdot x}f(x)\dif x, \quad\xi\in\mR^d,
$$
$$
\check f(x) :=(2 \pi)^{-d/2}\int_{\mR^d} \e^{i\xi\cdot x}f(\xi)\dif\xi, \quad x\in\mR^d.
$$
For every $f\in\sS'(\mR^d)$, the Fourier and the inverse Fourier transform are defined in the following way respectively,
\begin{align*}
\<\hat{f},\varphi\>:=\<f,\hat{\varphi}\>,\ \ \<\check{f},\varphi\>:=\<f,\check{\varphi}\>, \ \  \forall\varphi\in\sS(\mR^d).
\end{align*}
Let $\chi:\mR^{d}\to[0,1]$ be a radial smooth function with
\begin{align*}
\chi(\xi)=
\begin{cases}
1, & \ \  |\xi|\leq 1,\\
0, &\ \ |\xi|>3/2.
\end{cases}
\end{align*}
For $\xi \in \mR^d$, define $\psi(\xi):=\chi(\xi)-\chi(2\xi)$ and
\begin{align*}
\psi_j (\xi) := \begin{cases}
\chi(2\xi), &\ \ \text{for}\ \  j =-1,\\
\psi(2^{-j} \xi),& \ \ \text{for}\ \ j \in \mN_0.
\end{cases}
\end{align*}
Denote $B_r := \{\xi\in \mR^d ;  |\xi|\leq r\}$ for $r>0$. It is easy to see that  supp$\psi\subset B_{3/2}\backslash B_{1/2}$, and
\begin{align}\label{eq:SA00}
\sum_{j=-1}^{k}\psi_j(\xi)=\chi(2^{-k}\xi)\to 1,\ \ \hbox{as}\ \ k\to\infty.
\end{align}
For $j\ge -1$, the block operator $\cR_j$ is defined on $\sS'(\mR^d)$ by
\begin{align*} 
\cR_j f (x):=(\psi_j\hat{f})^{\check\,}(x)=\check\psi_j* f (x).
\end{align*}

Now we state the definition of Besov spaces.
\bd[Besov spaces]\label{iBesov}
For every $s\in\mR$ and $p,q\in[1,\infty]$, the Besov space $\bB_{p,q}^s(\mR^d)$ is defined by
$$
\bB_{p,q}^s(\mR^d):=\left\{f\in\sS'(\mR^d)\, \big| \, \|f\|_{\bB^s_{p,q}}:= \left[ \sum_{j \geq -1}\left( 2^{s j} \|\cR_j f\|_{p} \right)^q \right]^{1/q}  <\infty\right\}.
$$
If $p=q=\infty$, we denote
$$
\bC^s:=\bB_{\infty,\infty}^s(\mR^d).
$$
\ed 
\br
It is well known that when $s>0$ and $s\notin \mN$, 
\begin{align*}
    \|f\|_{\bC^s}\asymp\sum_{i=0}^{[s]}\|\nabla^if\|_\infty+\sup_{x\ne y}\frac{|\nabla^{[s]}f(x)-\nabla^{[s]}f(y)|}{|x-y|^{s-[s]}},
\end{align*}
which is the classical H\"older norm, where $[s]$ is the integer part of $s$. We refer to \cite[Remark 3.4]{WH23} for more details. 
\er

Recall the following Bernstein's inequality (cf. \cite[Lemma 2.1]{BCD11} and \cite[Lemma 2.2]{CZZ21}).

\bl[Bernstein's inequality]
For each $k\in\mN_0$, there is a constant $c=c(d,k)>0$ such that for all $j\ge-1$ and $1\leq p_1\leq p_2 \leq \infty$,
\begin{align}\label{Bern}
\|\nabla^k\cR_j f\|_{p_2}  \lesssim_{c}  2^{(k+ d (\frac{1}{p_1}-\frac{1}{p_2}))j}\|\cR_j f\|_{p_1},
\end{align}
and for some $c=c(d,\alpha,p_1)>0$
\begin{align}\label{Bern2}
\|\Delta^{\frac{\alpha}{2}}\cR_j f\|_{p_2}  \lesssim_{c}  2^{(\alpha+d(\frac{1}{p_1}-\frac{1}{p_2})) j}\|\cR_j f\|_{p_1}.
\end{align}
\el

Apart from the classical Bernstein inequality, there are also more general versions for
$$
\|\Delta^{\frac\alpha2}(|\cR_j f|^{\frac{p}{2}})\|_{2},\quad  \text{and}\quad \int_{\mR^d} |\cR_j f|^{p-2}\cR_j f \Delta^{\frac\alpha2}\cR_j f, \quad p\in[2,\infty),
$$  as established in \cite{CMZ07} and \cite[Lemma 3.2]{LZ22}. Here, we introduce a frequency-localized maximum principle corresponding to the case $p=\infty$ (see \cite[Lemma 4.7]{SX23} and \cite[Lemma 3.4]{WZ11}).
\begin{lemma}\label{LMP}
There is a constant $c_0>0$ such that for any $j\geq0$ and $u\in L^1$, $\cR_j u\in C_0$, and
\begin{equation*}
\inf_{x\in J(\cR_ju)}\left\{\sgn(\cR_j u(x))\cdot\left(-\Delta^{\frac\alpha2} \cR_j u(x)\right)\right\}
\geq c_02^{\alpha j}\|\cR_ju\|_{\infty},
\end{equation*}
where $J(\cR_j u):=\{x:|\cR_j u(x)|=\|\cR_j u\|_{\infty}\}$.
\end{lemma}
\begin{proof}
    Since $u\in L^1$, $\cR_j u$ is a bounded continuous function. Its Fourier transform is supported in $\{(1/2)2^j<|\xi|<(3/2)2^j\}$ and the Riemann-Lebesgue lemma implies $\lim_{x\to \infty} |\cR_j u(x)|=0$. Thus, the proof follows directly from \cite[Lemma 3.4]{WZ11} (see also \cite[Lemma 4.7]{SX23}).
\end{proof}

The following convolution inequality is well-known (see \cite[(2.13)]{HRZ23} for example).

\bl
For any $\beta,\beta_1,\beta_2\in\mR$ and $p,p_1,p_2\in[1,\infty]$ with 
\begin{align*}
    1+\tfrac1p=\tfrac1{p_1}+\tfrac1{p_2},
\end{align*}
it holds that
\begin{align}\label{ineq:con}
    \|f*g\|_{\bB^{\beta}_{p,\infty}}\le 5\|f\|_{\bB^{\beta_1}_{p_1,\infty}}\|g\|_{\bB^{\beta_2}_{p_2,\infty}}.
\end{align}
\el

We also introduce the following commutator estimate, which is proven in \cite[Lemma 2.3]{CZZ21}. Here for two operators $T_1$ and $T_2$, we use $[T_1,T_2]:=T_1T_2-T_2T_1$ to denote the commutator.
\begin{lemma}
Let $p,p_1,p_2,q_1,q_2\in[1,\infty]$ satisfying $\frac{1}{p}=\frac{1}{p_1}+\frac{1}{p_2}$
and $\frac{1}{q_1}+\frac{1}{q_2}=1$. Then for any $s_1\in(0,1)$ and $s_2\in[-s_1,0]$
there exists a constant $C=C(d,p,p_1,p_2,s_1,s_2)>0$ such that
\begin{equation}\label{ComEs}
\|[\cR_j,f]g\|_p\leq C2^{-j(s_1+s_2)}
\begin{cases}
\|f\|_{B_{p_1,\infty}^{s_1}}\|g\|_{p_2}, & \kappa_2=0,\\
\|f\|_{B_{p_1,\infty}^{s_1}}\|g\|_{B_{p_2,\infty}^{s_2}}, & s_1+s_2>0,\\
\|f\|_{B_{p_1,q_1}^{s_1}}\|g\|_{B_{p_2,q_2}^{s_2}}, & s_1+s_2=0.
\end{cases}
\end{equation}
\end{lemma}

\subsection{$\alpha$-stable processes}

Fix $\alpha\in(0,2)$.
Let $(L_t)_{t\geq0}$ be a rotationally invariant and symmetric $\alpha$-stable process
whose generator is the fractional Laplacian
\begin{equation}\label{fL}
\Delta^{\frac{\alpha}{2}}\phi(x)=\int_{\R^d}\left(\phi(x+z)-\phi(x)-z^{(\alpha)}\cdot\nabla\phi(x)\right)\nu(\dif z),
\quad \phi\in\sS(\R^d),
\end{equation}
where $z^{(\alpha)}:=z\b1_{\{\alpha\in[1,2)\}}\b1_{\{|z|\leq1\}}$ and the L\'evy measure
$$
\nu(\dif z)=\frac{c_{d,\alpha}}{|z|^{d+\alpha}}\dif z
$$
with some specific constant $c_{d,\alpha}>0$.
Let $N(\dif s,\dif z)$ be the associated Poisson random measure defined by
\begin{equation*}
N((0,t]\times A):=\sum_{s\in[0,t]}\b1_A\left(L_s-L_{s-}\right)
\end{equation*}
for all $A\in\sB(\R^d\setminus\{0\})$ and $t>0$. Then it follows from L\'evy--It\^o's
decomposition that
\begin{equation*}
L_t=\int_0^t\int_{|z|\leq1}z\widetilde{N}(\dif s,\dif z)
+\int_0^t\int_{|z|>1}zN(\dif s,\dif z),
\end{equation*}
where $\widetilde{N}(\dif s,\dif z):=N(\dif s,\dif z)-\nu(\dif z)\dif s$ is
the compensated Poisson random measure.

\section{Pathwise uniqueness for SDEs}\label{Sec:3}
In this section, we consider the following SDE:
\begin{align}\label{SDE0}
    \dif X_t=B(t,X_t)\dif t+\dif L_t,
\end{align}
where $B\in L^q_T\bC^\gamma$ with some $\gamma>1-\frac\alpha2$ and $q>1$. 

The following result is from \cite[Theorem 1.1-(iii)]{TW25}.
\bt\label{thm:31}
Assume $\alpha\in(0,2)$. Let $T>0$ and $B\in L^q_T\bC^\gamma$ with some $\gamma>1-\frac\alpha2$ and
\begin{align}\label{00}
    \tfrac{\alpha}{q}- \gamma<\alpha-1.
\end{align}
Then for any initial data $X_0\in \sF_0$, there is a unique strong solution to \eqref{SDE0} on $[0,T]$. Moreover, the pathwise uniqueness holds. 
\et
\br
For any $\gamma>1-\frac{\alpha}{2}$, $q=2$ always satisfies \eqref{00}.
\er

\section{Regularity of the time marginal law}\label{Sec:4}
In this section, we consider the nonlinear FPE \eqref{in:FPE}. Under assumption \eqref{con} with $\beta > 1 - \alpha$, it is well known (see, e.g., \cite{HRZ23, HRW24}) that for any $t \ge 0$, the measure $\mu_t$ admits a density with respect to the Lebesgue measure, denoted by $\rho_t$. The density $\rho_t$ satisfies the nonlinear FPE:
\begin{align}\label{FPEr}
\partial_t \rho_t = \Delta^{\frac{\alpha}{2}} \rho_t - \operatorname{div} \big( (K * \rho_t) \rho_t \big).
\end{align}

Our goal in this section is to establish the following regularity estimate for $\rho_t$ under various assumptions and for different ranges of $\delta \ge 0$:
\begin{align}\label{rhot}
        \|\rho_t\|_{\bB^{\delta}_{1,\infty}}\le Ct^{-\frac{\delta}{\alpha}}.
    \end{align}
We now state the main result in this section:
\bt\label{thm41}
Let $\alpha\in(0,2)$ and $T>0$. Assume that condition \eqref{con} holds with $\beta>1-\alpha$. Then:

\begin{itemize}
    \item[(i)] 
    If $\alpha \in (1,2)$, then for any $\delta \ge 0$, there exists a constant $C > 0$ such that estimate \eqref{rhot} holds for all $t \in (0,T]$.
    \item[(ii)] 
   If $\alpha \in (0,1]$, then for any $\delta \in [0, \alpha )$, there exists a constant $C > 0$ such that \eqref{rhot} holds for all $t \in (0,T]$.
\end{itemize}
\et
\begin{proof} Since \eqref{rhot} is trivial when $\delta = 0$, we consider only the case $\delta > 0$ in what follows.

\vspace{1mm}

    (i) If $\alpha \in (1,2)$, then the result follows directly from \cite[Theorem 1.9, (B)-(i)]{HRZ23}.

\vspace{1mm}
    
(ii)  Let $\alpha\in(0,1]$ and $\beta>1-\alpha\ge0$. For any $t\in(0,T]$ and $\varphi\in C^\infty_0$, we consider the following PDE:
    \begin{align}\label{KPDE}
        \p_s u^t(s) =\Delta^{\frac\alpha2}u^t(s)+ B(t-s)\cdot\nabla u^t(s),\quad u^t(0)=\varphi,
    \end{align}
    where $B(t):=K(t)*\rho_t$.
    Taking $X$ be the unique weak solution to \eqref{DDSDE}, applying It\^o's formula to the process $s\to u(t-s,X_s)$, we have
\begin{align*}
    \dif u^t(t-s,X_s)=&\left(-\p_s u^t(t-s)-\Delta^{\frac\alpha2}u^t(t-s)+ B(s)\cdot\nabla u^t(t-s) \right)(X_s)\dif s\\
    &+\int_{\mR^d}u(t-s,X_{s-}+z)-u(t-s,X_{s-})\tilde{N}(\dif s,\dif z)\\
    =&\int_{\mR^d}u(t-s,X_{s-}+z)-u(t-s,X_{s-})\tilde{N}(\dif s,\dif z),
\end{align*}
    which implies that
    \begin{align}\label{0815:00}
       \langle \varphi,\rho_t\rangle=\mE \varphi(X_t)= \mE u^t(0,X_t)=\mE u^t(t,X_0).
    \end{align}
     Then, based on \cite[Proposition 2.76]{BCD11}, 
     we have
     \begin{align}\label{in:00}
         \|\rho_t\|_{\bB^{\delta}_{1,\infty}}\le \|\rho_t\|_{\bB^{\delta}_{1,1}}\lesssim \sup_{\varphi\in C^\infty_0}\frac{|\langle \varphi,\rho_t\rangle|}{\|\varphi\|_{\bC^{-\delta}}}\le \sup_{\varphi\in C^\infty_0}\frac{\|u^t(t)\|_\infty}{\|\varphi\|_{\bC^{-\delta}}}.
     \end{align}
     Thus, to show \eqref{rhot}, we only need to give 
     \begin{align}\label{01}
         \|u^t(t)\|_\infty\lesssim t^{-\frac{\delta}{\alpha}}\|\varphi\|_{\bC^{-\delta}},
     \end{align}
     for any $\delta\in(0,\alpha)$.   To establish \eqref{01}, by a standard argument, we may assume $K\in L^\infty_T C^\infty_b$ and show that \eqref{01} holds with a constant depending only on $c_K:=\|K\|_{L^\infty_T\bC^\beta}$.


Applying $\cR_j$ to act both sides of \eqref{KPDE}, we have   \begin{align*}
    \p_s \cR_j u^t &=\Delta^{\frac{\alpha}{2}}\cR_j u^t +\cR_j(B(t-s)\cdot\nabla u^t)\\
    &=\Delta^{\frac{\alpha}{2}}\cR_j u^t +B(t-s)\cdot\nabla\cR_j u^t+[\cR_j,B(t-s)]\cdot\nabla u^t.
\end{align*} 
Since $B\in L^\infty_T\bC^\infty_b$ and $\varphi\in L^1$, it is well-known that $u\in L^\infty_T L^1$ (see \cite[Lemma C.1]{HRZ23}). Based on Lemma \ref{LMP}, $\cR_j u^t(s)\in C_0$, and we let $x_{s,j}$ be the point such that $\cR_j u^t(s)$ reaches its $L^\infty$-norm. Without loss of generality, we may assume that $\cR_j u^t (s,x_{s,j})>0$. Otherwise, we can consider the equation for $-u^t$. 
Then Lemma \ref{LMP} yields that for $j\ge0$,
\begin{align*}
    \Delta^{\frac{\alpha}{2}}\cR_ju^t(s,x_{s,j})\le -c_0 2^{\alpha j}\|\cR_j u^t(s)\|_\infty.
\end{align*}
Moreover, for $j=-1$, \eqref{Bern2} yields that
\begin{align*}
    \Delta^{\frac{\alpha}{2}}\cR_{-1}u^t(s,x_{s,j})\le \|\Delta^{\frac{\alpha}{2}}\cR_{-1}u^t(s)\|_\infty\le c_1 \|\cR_{-1} u^t(s)\|_\infty
\end{align*}
with some constant $c_1=c_1(d,\alpha)>0$.
Noting that $B(t-s,x_{s,j})\cdot\nabla\cR_j u^t(s,x_{s,j})=0$, we have for $j\ge -1$,
\begin{align*}
    \p_s \cR_j u^t(s,x_{s,j}) \le \left(-c_0 2^{\alpha j}+c_1\right)\|\cR_j u^t(s)\|_\infty +\|[\cR_j,B(t-s)]\cdot\nabla u^t(s)\|_\infty.
\end{align*}
    Thus, in view of \cite[Lemma 3.2]{WZ11}, we have 
    \begin{align}\label{S4:06}
        \p_s\|\cR_j u^t(s)\|_\infty=\p_s \cR_j u^t(s,x_{s,j})\lesssim (1-2^{\alpha j})\|\cR_j u^t(s)\|_\infty+\|[B(t-s),\cR_j]\nabla u^t(s)\|_\infty.
    \end{align}
     By \cite[Lemma 2.3]{CZZ21}, we have for any $\vartheta\in[0,1-\beta)$,
    \begin{align*}
        \|[B(t-s),\cR_j]\nabla u^t(s)\|_\infty\lesssim \|B(t-s)\|_{\bC^{\beta+\vartheta}}\|\nabla u^t(s)\|_{\bB^{-\beta-\vartheta}_{\infty,1}}\lesssim \|B(t-s)\|_{\bC^{\beta+\vartheta}}\|u^t(s)\|_{\bB^{1-\beta-\vartheta}_{\infty,1}},
    \end{align*}
    which by applying Gronwall's inequality to \eqref{S4:06} implies that
    \begin{align}\label{S4:07}
       \|\cR_j u^t(s)\|_\infty\lesssim e^{-c2^{-\alpha j}s}\|\cR_j\varphi\|_\infty+\int_0^s  e^{-c2^{-\alpha j}(s-r)}\|B(t-r)\|_{\bC^{\beta+\vartheta}}\|u^t(r)\|_{\bB^{1-\beta-\vartheta}_{\infty,1}}\dif r.
    \end{align}
    Then, noting that
  \begin{align}\label{S4:11}
e^{-x} \lesssim_\theta x^{-\theta},\quad \text{for $x > 0$,}    \end{align}  
    for any $\delta>0$ and $\theta>1-\beta-\vartheta$, we have
    \begin{align}\label{0813:01}
       2^{j\theta}\|\cR_j u^t(s)\|_\infty\lesssim s^{-\frac{\theta+\delta}{\alpha}}2^{-\delta j}\|\cR_j\varphi\|_\infty+\int_0^s  (s-r)^{-\frac{\theta}{\alpha}}\|B(t-r)\|_{\bC^{\beta+\vartheta}}\|u^t(r)\|_{\bB^{1-\beta-\vartheta}_{\infty,1}}\dif r,
    \end{align}
    which, after taking the supremum over $j$ yields
    \begin{align}\label{0813:00}
       \|u^t(s)\|_{\bB^{1-\beta-\vartheta}_{\infty,1}}\lesssim \|u^t(s)\|_{\bC^{\theta}}\lesssim s^{-\frac{\theta+\delta}{\alpha}}\|\varphi\|_{\bC^{-\delta}}+\int_0^s  (s-r)^{-\frac{\theta}{\alpha}}\|B(t-r)\|_{\bC^{\beta+\vartheta}}\|u^t(r)\|_{\bB^{1-\beta-\vartheta}_{\infty,1}}\dif r.
    \end{align}

Now, we prove \eqref{01} for $\delta\in(0,\alpha)$ by induction:
\begin{itemize}
    \item[{\bf(Step 1)}] We first show that \eqref{01} holds for all $\delta\in(0,\alpha+\beta-1)$;
\item[{\bf(Step 2)}]We then prove that if \eqref{01} holds for some  $\vartheta\in(0,1-\beta)$, it also holds for all $\delta\in(0,\alpha+\beta-1+\vartheta)$;
\item[{\bf(Step 3)}] Finally, we conclude that \eqref{01} holds for all $\delta\in(0,\alpha)$.
\end{itemize}

\vspace{2mm}

{\bf (Step 1):} We let $\vartheta=0$ and note that from \eqref{ineq:con}, it holds that
\begin{align*}
    \|B(t-r)\|_{\bC^{\beta+\vartheta}}\lesssim \|K\|_{L^\infty_T\bC^\beta}\|\rho_{t-r}\|_{\bB^0_{1,\infty}}\lesssim c_K,
\end{align*}
    which by applying Gronwall’s inequality of Volterra type (see \cite{Zh10}) to \eqref{0813:00}, yields that for any $\delta\in (0,\alpha+\beta-1)$, and $\theta\in(1-\beta,\alpha-\delta)$
 \begin{align*}
       \|u^t(s)\|_{\bB^{1-\beta}_{\infty,1}}\lesssim_{c_K} s^{-\frac{\theta+\delta}{\alpha}}\|\varphi\|_{\bB^{-\delta}_{\infty,\infty}}.
    \end{align*}
    Applying it to \eqref{S4:07} with $\vartheta=0$, we have for any $\eps\in(0,\alpha-\theta-\delta)$, 
    \begin{align}\label{S4:0700}
    \begin{split}
              \|\cR_j u^t(t)\|_\infty&\lesssim \|\varphi\|_{\bB^{-\delta}_{\infty,\infty}}\left(e^{-c2^{-\alpha j}t}2^{\delta j}+\int_0^t  e^{-c2^{-\alpha j}(t-r)}r^{-\frac{\theta+\delta}{\alpha}}\dif r\right)\\
       &\lesssim \|\varphi\|_{\bB^{-\delta}_{\infty,\infty}}\left(e^{-c2^{-\alpha j}t}2^{\delta j}+2^{-\eps j}\int_0^t (t-r)^{-\frac{\eps}{\alpha}}r^{-\frac{\theta+\delta}{\alpha}}\dif r\right).
         \end{split}
    \end{align}
    It follows from a change of variable that
    \begin{align*}
        \sum_{j\ge -1} e^{-c2^{-\alpha j}t}2^{\delta j}\lesssim \sum_{j\ge -1}\int_{2^{j}}^{2^{j+1}} e^{-cx^\alpha t/2} x^{\delta-1}\dif x\le t^{-\frac{\delta}{\alpha}}\int_0^\infty e^{-cx^\alpha/2} x^{\delta-1}\dif x,
    \end{align*}
    which implies that for any $\delta\in(0,\alpha+\beta-1)$
    \begin{align*}
        \|u^t(t)\|_{\infty}\lesssim \sum_{j\ge -1} \|\cR_j u^t(t)\|_\infty\lesssim t^{-\frac{\delta}{\alpha}}\|\varphi\|_{\bB^{-\delta}_{\infty,\infty}}.
    \end{align*}
    This establishes \eqref{01} for any $\delta \in (0, \alpha + \beta - 1)$. 

    \vspace{1mm}
    {\bf (Step 2):} We assume that \eqref{01} holds for some $\vartheta\in(0,1-\beta)$. Then based on \eqref{in:00}, we have \eqref{rhot} holds for $\delta=\vartheta$, which by \eqref{ineq:con} gives for any $\vartheta\in(0,\alpha+\beta-1)$,
    \begin{align*}
        \|B(t-r)\|_{\bC^\beta+\vartheta}\lesssim\|K\|_{L^\infty_T\bC^\beta}\|\rho_{t-r}\|_{\bB^\vartheta_{1,\infty}}\lesssim (t-r)^{-\frac{\vartheta}{\alpha}}c_K.
    \end{align*}
Substitute it into \eqref{0813:00}, we have for any $\delta>0$ and $\theta>1-\beta-\vartheta$, 
    \begin{align}\label{0813:01}
       2^{j\theta}\|\cR_j u^t(s)\|_\infty\lesssim_{c_K} s^{-\frac{\theta+\delta}{\alpha}}2^{-\delta j}\|\cR_j\varphi\|_\infty+\int_0^s  (s-r)^{-\frac{\theta}{\alpha}}(t-r)^{-\frac{\vartheta}{\alpha}}\|u^t(r)\|_{\bB^{1-\beta-\vartheta}_{\infty,1}}\dif r,
    \end{align}
which, by taking the supremum over $j$ imposes that
\begin{align*}
     \|u^t(s)\|_{\bB^{1-\beta-\vartheta}_{\infty,1}}\lesssim\|u^t(s)\|_{\bC^\theta}\lesssim s^{-\frac{\theta+\delta}{\alpha}}\|\varphi\|_{\bC^{-\delta}}+\int_0^s  (s-r)^{-\frac{\theta+\vartheta}{\alpha}}(\|u^t(r)\|_{\bB^{1-\beta-\vartheta}_{\infty,1}}\dif r.
\end{align*}
Then, for any $\delta\in(0,\alpha+\beta-1+\vartheta)$ and $\theta\in(1-\beta-\vartheta,\alpha-\delta\vee \vartheta)$, by applying Gronwall’s inequality of Volterra type (see \cite{Zh10}), we have
\begin{align*}
    \|u^t(s)\|_{\bB^{1-\beta-\vartheta}_{\infty,1}}\lesssim s^{-\frac{\theta+\delta}{\alpha}}\|\varphi\|_{\bC^{-\delta}}.
\end{align*}
Then similar as \eqref{S4:0700}, 
 applying it to \eqref{S4:07}, we have 
    \begin{align*}
       \|u^t(t)\|_\infty&\le\sum_{j}\|\cR_j u^t(t)\|_\infty
       \lesssim \|\varphi\|_{\bB^{-\delta}_{\infty,\infty}}\le\sum_{j}\left(e^{-c2^{-\alpha j}t}2^{\delta j}+2^{-\eps}\int_0^t (t-r)^{-\frac{\vartheta+\eps}{\alpha}}r^{-\frac{\theta+\delta}{\alpha}}\dif r\right)\\
       &\lesssim \|\varphi\|_{\bB^{-\delta}_{\infty,\infty}}t^{-\frac{\delta}{\alpha}},
    \end{align*}
where we used the fact $\vartheta+\theta<\alpha$ and 
\begin{align*}
    \int_0^t (t-r)^{-\frac{\vartheta+\eps}{\alpha}}r^{-\frac{\theta+\delta}{\alpha}}\dif r\lesssim t^{-\frac{\vartheta+\theta-\alpha+\eps+\delta}{\alpha}}\lesssim t^{-\frac{\delta}{\alpha}}
\end{align*}
with small $\eps>0$. Then we have \eqref{01} for $\delta\in(0,\alpha+\beta-1+\vartheta)$ and finish the proof of Step 2.

    \vspace{2mm}

    {\bf (Step 3):} We begin by choosing $\delta_0 \in (0,\alpha+\beta-1)$.  
By Step~1, \eqref{01} holds for all $\delta \in (0,\delta_0)$.  
Applying Step~2, we obtain that \eqref{01} holds for all $\delta \in (0,\delta_1)$, where  
$$
\delta_1 := \alpha+\beta-1 + \min\{\delta_0,\, 1-\beta\}.
$$  
Repeating this argument inductively, Step~2 yields that \eqref{01} holds for all  
$$
\delta \in \bigcup_{k=0}^\infty (0,\delta_k),
$$  
where  
$$
\delta_k := \alpha+\beta-1 + \min\{\delta_{k-1},\, 1-\beta\}, \quad k \ge 2.
$$  
Since  
$$
\lim_{k \to \infty} \delta_k = \alpha+\beta-1 + (1-\beta) = \alpha,
$$  
we conclude that \eqref{01} holds for all $\delta \in (0,\alpha)$.  
This completes the proof.
\end{proof}

\section{Proof of Theorem \ref{in:main}}\label{Sec:5}
Since the weak uniqueness holds for $\beta > 1 - \alpha$ (see \cite{CJM25} for $\alpha > 1$ and \cite{HRW24} for $\alpha \leq 1$), by the Yamada–Watanabe theorem, it suffices to prove pathwise uniqueness for the linearized SDE \eqref{SDE0} with 
$$
B(t,x):=(K*\mu_{X_t})(t,x)=(K*\rho_{t})(t,x).
$$
Based on \eqref{ineq:con} and Theorem \ref{thm41}, for any $\eps \in (0, \beta + \alpha - 1)$, by taking $\delta=\frac{\alpha}{2}-\eps$ in \eqref{rhot}, we have
\begin{align}\label{proof:00}
    \|B(t)\|_{\bC^{1-\frac{\alpha}{2}+\frac{\eps}{2}}}\lesssim \|K\|_{L^\infty_T\bC^{1-\alpha+\eps}}\|\rho_t\|_{\bB^{\frac{\alpha-\eps}{2}}_{1,\infty}}\lesssim \|K\|_{L^\infty_T\bC^{\beta}}t^{-\frac{1}{2}+\frac{\eps}{\alpha}},
\end{align}
which implies that
\begin{align*}
    B\in L^2([0,T];\bC^{1-\frac{\alpha}{2}+\frac\eps2}).
\end{align*}
    Then, pathwise uniqueness follows from Theorem \ref{thm:31}, and this completes the proof.



\end{document}